\documentclass[12pt]{article}
\usepackage{amsmath}
\usepackage{amssymb}
\usepackage{amsthm}
\usepackage[mathscr]{eucal}
\usepackage{latexsym}
\usepackage{amscd}
\usepackage{epic}
\usepackage{epsfig}

\begin{document}
\baselineskip=15pt

\voffset -1truecm
\oddsidemargin0.3truecm

\theoremstyle{plain}
\swapnumbers\newtheorem{lema}{Lemma}[section]
\newtheorem{prop}[lema]{Proposition}
\newtheorem{coro}[lema]{Corollary}
\newtheorem{teor}[lema]{Theorem}
\newtheorem{demo}[lema]{Proof of Theorem 1}
\newtheorem{ejem}[lema]{Example}
\newtheorem{obs}[lema]{Remark}

\renewcommand{\refname}{\large\bf References}
\renewcommand{\thefootnote}{\fnsymbol{footnote}}

\def\natural{{\mathbb N}}
\def\entero{{\mathbb Z}}
\def\racional{{\mathbb Q}}
\def\real{{\mathbb R}}
\def\complejo{{\mathbb C}}

\def\vector#1#2{({#1}_1,\dots,{#1}_{#2})}
\def\flecha{\longrightarrow}
\def\asocia{\longmapsto}
\def\sii{\Longleftrightarrow}
\def\implica{\Longrightarrow}
\def\conjunto#1{\boldsymbol{[}\,{#1}\,\boldsymbol{]}}
\def\wt#1{\widetilde{#1}}
\def\ol#1{\overline{#1}}

\def\la{\lambda}
\def\longi#1{\ell({#1})}
\def\lintm{\lambda\cap\mu}
\def\domina{\succeq}
\def\rectan#1#2{({#1})^{#2}}
\def\liri{Littlewood-Richardson\ }
\def\lirim{Littlewood-Richardson multitableau\ }
\def\lirimx{Littlewood-Richardson multitableaux\ }
\def\sbst#1#2{{\sf SBST}({#1},{#2})}
\def\coefliri#1#2#3{c^{\,#1}_{{#2}\,{#3}}}
\def\coefliril#1#2{c^{\,\lambda}_{{#1}\,{#2}}}
\def\ro1r{(\rho(1), \dots ,\rho(r))}
\def\rosec{\rho(1)\vdash \pi_1, \dots , \rho(r)\vdash \pi_r}
\def\clirir#1{c^{#1}_{\ro1r}}
\def\kos#1#2{K_{{#1}\,{#2}}}
\def\multiliri#1{{\sf lr}(\la, \mu; {#1})}
\def\mmultiliri#1{{\sf LR}(\la, \mu; {#1})}
\def\alt#1{{\sf ht}({#1})}
\def\signo#1{{\sf sign}({#1})}
\def\sec#1#2{\varnothing = {#1}(0) \subset {#1}(1) \subset \cdots
\subset {#1}({#2}) }

\def\cara#1{\chi^{#1}}
\def\permu#1{\phi^{#1}}
\def\coefi{{\sf k}(\la,\mu,\nu)}
\def\kron{\chi^\lambda\otimes\chi^\mu}

\def\gruposim#1{{\sf S}_{#1}}

\def\coeflirim#1#2{c^{\,\mu}_{{#1}\,{#2}}}
\def\partil#1{\lambda({#1})}
\def\partim#1{\mu({#1})}
\def\partilintm#1{\lambda({#1})\cap\mu({#1})}
\def\pareja{(\lambda,\,\mu)}
\def\diferencial{\lambda-\lambda\cap\mu}
\def\diferenciam{\mu-\lambda\cap\mu}

%%%%%%%%%%%%%%%%%%%%%%%%%%%%%%%%%%%%%%%%%%%%%%%%%%%%%
\def\tablau{\tabla{ \ \\}}
\def\tabladh{\tabla{ \ & \ \\}}
\def\tabladv{\tabla{ \ \\ \ \\}}
\def\tablath{\tabla{ \ & \ & \ \\ }}
\def\tablatv{\tabla{ \ \\ \ \\ \ \\}}
\def\tablatdu{\tabla{ \ & \ \\ \ \\}}
\def\tablatud{\tabla{ & \ \\ \ & \ \\ }}
\def\tablauuu{\tablau\sqcup\tablau\sqcup\tablau}

%%%%%%%%Tableaux Macros%%%%%%%%%%%%%%%
%%%%%%%%%%%%%%%%%%%%%%%%%%%%%%%%

\setlength\unitlength{0.08em}
\savebox0{\rule[-2\unitlength]{0pt}{10\unitlength}%
\begin{picture}(10,10)(0,2)
\put(0,0){\line(0,1){10}}
\put(0,10){\line(1,0){10}}
\put(10,0){\line(0,1){10}}
\put(0,0){\line(1,0){10}}
\end{picture}}

\newlength\cellsize \setlength\cellsize{18\unitlength}
\savebox2{%
\begin{picture}(18,18)
\put(0,0){\line(1,0){18}}
\put(0,0){\line(0,1){18}}
\put(18,0){\line(0,1){18}}
\put(0,18){\line(1,0){18}}
\end{picture}}
\newcommand\cellify[1]{\def\thearg{#1}\def\nothing{}%
\ifx\thearg\nothing
\vrule width0pt height\cellsize depth0pt\else
\hbox to 0pt{\usebox2\hss}\fi%
\vbox to 18\unitlength{
\vss
\hbox to 18\unitlength{\hss$#1$\hss}
\vss}}
\newcommand\tableau[1]{\vtop{\let\\=\cr
\setlength\baselineskip{-16000pt}
\setlength\lineskiplimit{16000pt}
\setlength\lineskip{0pt}
\halign{&\cellify{##}\cr#1\crcr}}}
\savebox3{%
\begin{picture}(15,15)
\put(0,0){\line(1,0){15}}
\put(0,0){\line(0,1){15}}
\put(15,0){\line(0,1){15}}
\put(0,15){\line(1,0){15}}
\end{picture}}
\newcommand\expath[1]{%
\hbox to 0pt{\usebox3\hss}%
\vbox to 15\unitlength{
\vss
\hbox to 15\unitlength{\hss$#1$\hss}
\vss}}
%%%%%%%%%%%%%%%%%%%%%%%%%%%%%%%%%%%%%%%%%%%%
\newlength\celulita \setlength\celulita{5\unitlength}
\savebox3{%
\begin{picture}(5,5)
\put(0,0){\line(1,0){5}}
\put(0,0){\line(0,1){5}}
\put(5,0){\line(0,1){5}}
\put(0,5){\line(1,0){5}}
\end{picture}}
\newcommand\celificar[1]{\def\thearg{#1}\def\nothing{}%
\ifx\thearg\nothing
\vrule width0pt height\celulita depth0pt\else
\hbox to 0pt{\usebox3\hss}\fi%
\vbox to 5\unitlength{
\vss
\hbox to 5\unitlength{\hss$#1$\hss}
\vss}}
\newcommand\tablita[1]{\vtop{\let\\=\cr
\setlength\baselineskip{-16000pt}
\setlength\lineskiplimit{16000pt}
\setlength\lineskip{0pt}
\halign{&\celificar{##}\cr#1\crcr}}}

%%%%%%%%%%%%%%%%%%%%%%%%%%%%%%%%%%%%%%%%%%%%%
\newlength\cuadro \setlength\cuadro{7\unitlength}
\savebox4{%
\begin{picture}(7,7)
\put(0,0){\line(1,0){7}}
\put(0,0){\line(0,1){7}}
\put(7,0){\line(0,1){7}}
\put(0,7){\line(1,0){7}}
\end{picture}}
\newcommand\cuadrificar[1]{\def\thearg{#1}\def\nothing{}%
\ifx\thearg\nothing
\vrule width0pt height\cuadro depth0pt\else
\hbox to 0pt{\usebox4\hss}\fi%
\vbox to 7\unitlength{
\vss
\hbox to 7\unitlength{\hss$#1$\hss}
\vss}}
\newcommand\tabla[1]{\vtop{\let\\=\cr
\setlength\baselineskip{-16000pt}
\setlength\lineskiplimit{16000pt}
\setlength\lineskip{0pt}
\halign{&\cuadrificar{##}\cr#1\crcr}}}

%%%%%%%%%%%%%%%%%%%%%%%%%%%%%%%%%%%%%%%%%%%%%

\begin{centering}
{\Large\bf A stability property for coefficients in}\\[.2cm]
{\Large\bf Kronecker products of complex {\mathversion{bold}$S_n$} characters}\\[1cm]
{\Large\sf Ernesto Vallejo\footnotemark[1]}\\[.1cm]
Universidad Nacional Aut\'onoma de M\'exico\\
Instituto de Matem\'aticas, Unidad Morelia\\
Apartado Postal 61-3, Xangari\\
58089 Morelia, Mich., MEXICO\\
e-mail: {\tt vallejo@matmor.unam.mx}\\[.4cm]
\end{centering}
\footnotetext[1]{Supported by CONACYT-Mexico, 47086-F and UNAM-DGAPA IN103508}

\vskip 2pc
\begin{abstract}
In this note we make explicit a stability property for
Kronecker coefficients that is implicit in a theorem of
Y.~Dvir.
Even in the simplest nontrivial case this property was overlooked
despite of the work of several authors.
As an application we give a new formula for some Kronecker
coefficients.
\smallskip

\end{abstract}

\section{Introduction}

Let $\la$, $\mu$, $\nu$ be partitions of a positive integer $m$
and let $\cara\la$, $\cara\mu$, $\cara\nu$ be their corresponding
complex irreducible characters of the symmetric group $\gruposim m$.
It is a long standing problem to give a satisfactory method for
computing the multiplicity
\begin{equation} \label{ecua:coefi}
\coefi := \langle \kron, \cara\nu \rangle
\end{equation}
of $\cara\nu$ in the Kronecker product $\kron$ of $\cara\la$ and
$\cara\mu$ (here $\langle \cdot, \cdot \rangle$ denotes the inner
product of complex characters).
Via the Frobenius map, $\coefi$ is equal to the
multiplicity of the Schur function
$s_\nu$ in the internal product of Schur functions
$s_\la \ast s_\mu$, namely
\begin{equation*}
\coefi = \langle s_\la \ast s_\mu, s_\nu \rangle\, ,
\end{equation*}
where $\langle \cdot, \cdot \rangle$ denotes the scalar product of
symmetric functions.

The first stability property for Kronecker coefficients was
observed by F.~Murnaghan without proof in~\cite{mur}.
This property can be stated in the following way:
Let $\ol\la$, $\ol\mu$, $\ol\nu$ be partitions of $a$, $b$, $c$,
respectively.
Define $\la(n):=(n-a,\ol\la)$, $\mu(n):=(n-b,\ol\mu)$,
$\nu(n):=(n-c,\ol\nu)$.
Then the coefficient ${\sf k}(\la(n), \mu(n), \nu(n))$ is constant
for all $n$ bigger than some integer $N(\ol\la, \ol\mu, \ol\nu)$.
Complete proofs of this property were given by
M.~Brion~\cite{bri} using algebraic geometry and
E.~Vallejo~\cite{vejc} using combinatorics of Young tableaux.
Both proofs give different lower bounds
$N(\ol\la, \ol\mu, \ol\nu)$ for the stability of
${\sf k}(\la(n), \mu(n), \nu(n))$, for all partitions
$\ol\la$, $\ol\mu$, $\ol\nu$.
C. Ballantine and R. Orellana~\cite{baor1} gave an improvement of one of
these lower bounds for a particular case.

Here we make explicit another stability property for Kronecker
coefficients that is implicit in the work of Y.~Dvir
(Theorem~$2.4^\prime$ in~\cite{dvir}).
This property can be stated as follows:
Let $p$, $q$, $r$ be positive integers such that $p=qr$.
Let $\la = \vector{\la}p$, $\mu = \vector{\mu}q$, $\nu =\vector{\nu}r$
be partitions of  some nonnegative integer $m$
satisfying $\longi{\la} \le p$, $\longi{\mu} \le q$,
$\longi\nu \le r$, that is, some parts of $\la$, $\mu$ and $\nu$ could be zero.
For any positive integers $t$ and $n$ let $\rectan tn$ denote the vector
$(t, \dots, t)\in\natural^n$.
Then we have

\bigskip
{\bf Theorem~\ref{teor:principal}.}
With the above notation
\[
\coefi = {\sf k}(\la + \rectan tp, \mu + \rectan{rt}q, \nu + \rectan{qt}r)\, .
\]

It should be noted that even in the simplest nontrivial case, when
$q=2=r$ and $p=4$, this property was overlooked despite of the work
of several authors~\cite{baor1,baor2,rewh,ros}.
In this situation Remmel and Whitehead noticed
(Theorems~3.1 and~3.2 in \cite{rewh}) that the coefficient
$\coefi$ has a much simpler formula if $\la_3=\la_4$.
The main theorem provides and explanation for that.
We also obtain a new formula for $\coefi$ in this case.

This note is organized as follows.
Section~\ref{sec:par-tab} contains the definitions and notation about
partitions needed in this paper.
In Section~\ref{sec:teor-principal} we give the proof of the main
theorem.
Section~\ref{sec:multita} deals with the Kronecker coefficient $\coefi$
when $\longi\la=\longi\mu\longi\nu$.
In particular, we give, in this case, a new vanishing condition.
Finally, in Section~\ref{sec:aplica} we give an application
of the main theorem.

\section{Partitions} \label{sec:par-tab}

In this section we recall the notation about
partitions needed in this paper.
See for example~\cite{jake, mac, sag, stan}.

For any nonnegative integer $n$ let $\conjunto n:=\{1, \dots, n\}$.
A {\em partition} is a vector $\la = \vector{\la}p$
of nonnegative integers arranged in decreasing order
$\la_1 \ge \cdots \ge \la_p$.
We consider two partitions equal if they differ by a string
of zeros at the end.
For example $(3,2,1)$ and $(3,2,1,0,0)$ represent the same
partition.
The {\em length} of $\la$, denoted by $\longi{\la}$, is the number
of positive parts of $\la$.
The {\em size} of $\la$, denoted by $|\la |$ is the sum of its parts;
if $|\la|=m$, we say that $\la$ is a partition of $m$ and denote
it by $\la\vdash m$.
The partition conjugate to $\la$ is denoted by $\la^\prime$.
A {\em composition} of $m$ is a vector $\pi= \vector \pi r$ of
positive integers such that $\sum_{i=1}^r \pi_i = m$.

The {\em diagram} of $\la$, also denoted by $\la$,
is the set of pairs of positive integers
\[
\la = \{\, ( i, j) \mid  i \in \conjunto p,\  j \in \conjunto{\la_i} \,\}.
\]
The identification of $\la$ with its diagram permits us to use
set theoretic notation for partitions.
If $\delta$ is another partition and $\delta \subseteq \la$,
we denote by $\lambda/ \delta$ the {\em skew diagram} consisting of the
pairs in $\la$ that are not in $\delta$, and by
$\vert\lambda / \delta\vert$ its cardinality.
If $\mu$ is another partition, then $\la\cap\mu$
denotes the set theoretic intersection of $\la$ and $\mu$.

\section{Main theorem} \label{sec:teor-principal}

\begin{teor} \label{teor:principal}
Let $\la$, $\mu$, $\nu$ be partitions of some integer $m$.
Let $p$, $q$, $r$ be integers such that $p\ge \longi\la$,
$q\ge \longi\mu$, $r\ge\longi\nu$ and $p=qr$.
Then for any positive integer $t$ we have
\[
\coefi = {\sf k}(\la + \rectan tp, \mu + \rectan{rt}q, \nu + \rectan{qt}r)\, .
\]
\end{teor}

The proof of the main theorem will follow from Dvir's theorem

\begin{teor} \cite[Theorem 2.4$^\prime$]{dvir}
Let $\la$, $\mu$, $\nu$ be partitions of $n$ such that
$\longi\nu = |\la\cap\mu^\prime |$.
Let $l=\longi\nu$ and $\rho=\nu - (1^l)$.
Then
\[
\coefi = \langle \cara{\la/\la\cap\mu^\prime}\otimes
\cara{\mu/\la^\prime\cap\mu}, \cara\rho \rangle\, .
\]
\end{teor}

\begin{proof}[Proof of theorem~\ref{teor:principal}]
It is enough to prove the theorem for $t=1$.
The general case follows by repeated application of
the particular case.
Let $\alpha= \la + (1)^p$, $\beta = \mu + (r)^q$
and $\gamma = \nu + (q)^r$.
Then $\beta\cap\gamma^\prime =(r^q)$.
In particular, $|\beta\cap\gamma^\prime| = p = \longi\alpha$.
So, we have $\beta/\beta\cap\gamma^\prime = \mu$
and $\gamma/\beta^\prime\cap\gamma = \nu$.
Thus, by Dvir's theorem, we have
\[
{\sf k}(\beta,\gamma,\alpha)= {\sf k}(\mu,\nu,\la)\, .
\]
The claim follows from the symmetry $\coefi = {\sf k}(\mu,\nu,\la)$
of Kronecker coefficients.
\end{proof}

\section{The case $\boldsymbol{\longi\la = \longi\mu \longi\nu}$} \label{sec:multita}

In this section we give a general result for the Kronecker coefficient
$\coefi$ when $\longi\la = \longi\mu \longi\nu$.
On the one hand it gives a new vanishing condition.
On the other hand, when this vanishing condition does not hold, it
reduces the computation of $\coefi$ to the computation of a simpler
Kronecker coefficient.

Let $m$ be a positive integer, $\la$, $\mu$ be partitions of $m$
and $\pi = \vector{\pi}r$ be a composition of $m$.
Let $\rho(i)\vdash\pi_i$ for $i\in \conjunto r$.
A sequence $T=(T_1, \dots , T_r)$ of tableaux is called a
{\em \liri multitableau} of {\em shape} $\la$, {\em content}
$\ro1r$ and {\em type} $\pi$ if

(1) there exists a sequence of partitions
\[
\varnothing=\la(0) \subset \la(1) \subset \cdots \subset \la(r) = \la
\]
such that $|\la(i)/\la(i-1)|=\pi_i$ for all $ i \in \conjunto r$, and

(2) $T_i$ is \liri tableau of shape $\la(i)/\la(i-1)$ and content
$\rho(i)$, for all $i\in \conjunto r$.

Let $\mmultiliri \pi$ denote the set of pairs $(S,T)$ of
\liri multitableaux of shape $(\la,\mu)$, same content and type $\pi$.
This means that $S=\vector Sr$ is a \liri multitableau of shape
$\la$, $T=\vector Tr$ is a \liri multitableau of shape $\mu$
and both $S_i$ and $T_i$ have the same content $\rho(i)$ for
some partition $\rho(i)$ of $\pi_i$, for all $ i\in \conjunto r$.
Let $\clirir \la$ denote the number of \liri multitableaux of shape $\la$
and content $\ro1r$ and let $\multiliri \pi$ denote the cardinality of
$\mmultiliri \pi$.
Then
\begin{equation*}
\multiliri \pi = \sum_{\rosec} \clirir \la \clirir \mu\, .
\end{equation*}

Similar numbers have already proved to be useful in the study of
minimal components, in the dominance order of partitions, of Kronecker
products~\cite{vjac}.

The number $\multiliri \pi$ can be described as an inner product of characters.
For this description we need the permutation character
$\permu\pi := {\rm Ind}_{\gruposim \pi}^{\gruposim m}( 1_{\pi})$,
namely, the induced character from
the trivial character of $\gruposim \pi =
\gruposim{\pi_1} \times \cdots \times \gruposim{\pi_r}$.
It follows from Frobenius reciprocity and the Littlewood-Richardson
rule that (see also~\cite[2.9.17]{jake})

\begin{lema} \label{lema:lr}
Let $\la$, $\mu$, $\pi$ be as above.
Then
\[
\multiliri \pi = \langle \kron, \permu\pi \rangle.
\]
\end{lema}

Since Young's rule and Lemma~\ref{lema:lr} imply that
$\multiliri \nu \ge \coefi$, then we have

\begin{coro} \label{coro:facil}
Let $\la$, $\mu$, $\nu$ be partitions of $m$.
If $\multiliri \nu =0$, then $\coefi =0$.
\end{coro}

\begin{lema} \label{lema:lrcero}
Let $\la$, $\mu$, $\nu$ be partitions of $m$ of lengths $p$, $q$,
$r$, respectively.
If $p=qr$, and $\mu_q < r\la_p$ or $\nu_r <
q\la_p$, then $\multiliri\nu =0$.
\end{lema}
\begin{proof}
We assume that $\multiliri\nu >0$ and show that
$\mu_q \ge r\la_p$ and $\nu_r \ge q\la_p$.
Let $(S,T)$ be an element in $ \mmultiliri \nu$ having
content $\ro1r$.
Since $T_i$ is contained in $\mu$, one has, by elementary properties of
Littlewood-Richardson tableaux,
that $\longi{\rho(i)} \le \longi\mu \le q$.
For any $i$, let $n_i$ be the number of squares of $S_i$ that are in
column $\la_p$ of $\la$, then $n_i \le q$.
We conclude that $p = n_1 + \cdots + n_r \le rq =p$.
Therefore $n_i =q = \longi{\rho(i)}$ for all $i$.
This forces that each $S_i$ contains a $j$ in the squares
$(j+(i-1)q, 1), \dots, (j+(i-1)q, \la_p)$ of $\la$,
for all $ j \in \conjunto q$.
So, $\rho(i)_j \ge \la_p$ for all $j$.
In particular, for $i=r$, since $S_r$ has $\nu_r$ squares,
one has $\nu_r  \ge q\la_p$.
Now, since $\longi\mu = q$, all entries of $T_i$ equal to $q$
must be in row $q$ of $\mu$.
Then $\mu_q \ge \rho(1)_q + \cdots + \rho(r)_q \ge r\la_p$.
The claim follows.
\end{proof}

\begin{coro} \label{coro:apliuno}
Let $\la$, $\mu$, $\nu$ be partitions of $m$ of length $p$, $q$, $r$,
respectively.
If $p=qr$, and $\mu_q < r\la_p$ or $\nu_r < q\la_p$, then
$ \coefi = 0$.
\end{coro}
\begin{proof}
This follows from Lemma~\ref{lema:lrcero} and Corollary~\ref{coro:facil}.
\end{proof}

Corollary~\ref{coro:apliuno} and Theorem~\ref{teor:principal} imply the
following

\begin{teor} \label{teor:peporqu}
Let $\la$, $\mu$, $\nu$ be  partitions of m of length $p$, $q$, $r$,
respectively.
Let $t=\la_p$ and assume $p=qr$, then we have

(1) If $\mu_q < rt$ or $\nu_r < qt$, then $\coefi = 0$.

(2) If $\mu_q \ge rt$ and $\nu_r \ge qt$, let
$\wt\la = \la - \rectan tp$, $\wt\mu = \mu - \rectan{rt}q$
and $\wt\nu = \nu -\rectan{qt}r$.
Then, $\coefi = {\sf k}(\wt\la,\wt\mu,\wt\nu)$.
\end{teor}

\section{Applications} \label{sec:aplica}

We conclude this paper with an application to the expansion of
$\cara\mu\otimes\cara\nu$ when $\longi\mu =2 = \longi\nu $.
It is well known that any component of
$\cara\mu\otimes\cara\nu$ corresponds to a partition of length at
most $|\mu\cap\nu^{\,\prime}|\le 4$, see
Satz~1 in~\cite{clme}, Theorem~1.6 in~\cite{dvir} or
Theorem~2.1 in~\cite{rewh}.
Even in this simple case a {\em nice} closed formula seems
unlikely to exist.
J.~Remmel and T.~Whitehead (Theorem~2.1 in~\cite{rewh}) gave a close,
though intricate, formula for $\coefi$ valid for any $\la$ of length at
most 4; M. Rosas (Theorem~1 in~\cite{ros}) gave a formula of
combinatorial nature for $\coefi$, which requires taking subtractions,
also valid for any $\la$ of length at most 4; C.~Ballantine and
R.~Orellana (Proposition~4.12 in~\cite{baor2}) gave a simpler
formula for $\coefi$, at the cost of assuming an extra condition on
$\la$.

Note that when $\longi\la =1$ the coefficient $\coefi$ is trivial to
compute.
For $\longi\la=2$  Remmel-Whitehead formula for $\coefi$ reduces
to a simpler one (Theorem~3.3 in \cite{rewh}).
This formula was recovered by Rosas in a different way
(Corollary~1 in~\cite{ros}).
So, the nontrivial cases are those for which $\longi\la = 3,4$.
Corollary~\ref{coro:dospordos} deals with the case of length 4.
On the one hand it gives a new vanishing condition.
On the other hand, when this vanishing condition does not hold, it
reduces the case of length 4 to the case of length 3.
Thus, this reduction would help to simplify the proofs of the formulas
given by Remmel-Whitehead and Rosas.

The following corollary is a particular case of Theorem~\ref{teor:peporqu}.

\begin{coro} \label{coro:dospordos}
Let $\la$, $\mu$, $\nu$ be a partitions of $m$ of length 4, 2, 2, respectively.
Let $t=\la_4$, then we have

(1) If $\mu_2 < 2t$ or $\nu_2 < 2t$, then $\coefi = 0$.

(2) If $\mu_2 \ge 2t$ and $\nu_2 \ge 2t$, let
$\wt\la = (\la_1 -t, \la_2-t, \la_3-t)$,
$\wt\mu = (\mu_1- 2t, \mu_2 -2t)$ and $\wt\nu = (\nu_1-2t, \nu_2 -2t)$.
Then, $\coefi = {\sf k}(\wt\la,\wt\mu,\wt\nu)$.
\end{coro}

Another observation of Remmel and Whitehead (Theorems~3.1 and~3.2
in~\cite{rewh}) is that their formula simplifies considerably in the
case $\la_3=\la_4$.
Corollary~\ref{coro:dospordos} explains this phenomenon since,
in this case, the computation of $\coefi$ reduces to the computation of a
Kronecker coefficient involving only three partitions of length
at most 2, which have a simple nice formula (Theorem~3.3 in~\cite{rewh}).
In fact, combining our result with this simple formula
we obtain a new one.
For completeness we record here Remmel-Whitehead formula in the equivalent
version of Rosas.

In the next theorems the notation $(y\ge x)$ means 1 if $y\ge x$
and 0 if $y\ngeq x$.

\begin{teor} \cite[Theorem~3.3]{rewh} \label{teor:rewh}
Let $\la$, $\mu$, $\nu$ be partitions of $m$ of length 2.
Let $x=\max\left(0, \left\lceil \tfrac{\nu_2+\mu_2+\la_2 -m}{2}\right\rceil \right)$ and
$y= \left\lceil \tfrac{\nu_2+\mu_2 -\la_2 +1}{2}\right\rceil$.
Assume $\nu_2 \le \mu_2 \le \la_2$.
Then
\[
\coefi = (y-x)(y\ge x)\, .
\]
\end{teor}

From Corollary~\ref{coro:dospordos} and Theorem~\ref{teor:rewh} we obtain

\begin{teor}
Let $\la$, $\mu$, $\nu$ be partitions of $m$ of length 4, 2, 2, respectively.
Suppose that $\la_3 = \la_4$ and that $2\la_3 \le \nu_2 \le \mu_2$.
Let $x=\max\left( 0, \left\lceil \tfrac{\nu_2+\mu_2+\la_2
-\la_3 -m}{2}\right\rceil\right)$,
$y=\left\lceil\tfrac{\nu_2 + \la_2 - \mu_2 - \la_3 + 1}{2}\right\rceil$
and
$z=\left\lceil\tfrac{\nu_2 + \mu_2 - \la_2 - 3\la_3 + 1}{2}\right\rceil$.
We have

(1) If $\la_2 + \la_3 \le \mu_2$, then $\coefi = (y-x)(y \ge x)$.

(2) If $\la_2 + \la_3 > \mu_2$, then $\coefi = (z-x)(z \ge x)$.
\end{teor}
\begin{proof}
Let $\wt\la = (\la_1 - \la_3, \la_2 - \la_3)$,
$\wt\mu = (\mu_1- 2\la_3, \mu_2 -2\la_3)$ and
$\wt\nu = (\nu_1-2\la_3, \nu_2 -2\la_3)$.
These are partitions of $m-4\la_3$.
Then, by Corollary~\ref{coro:dospordos},
$\coefi = {\sf k}(\wt\la,\wt\mu,\wt\nu)$.
Since $\longi{\wt\la} = \longi{\wt\mu} = \longi{\wt\nu}=2$,
we can apply Theorem~\ref{teor:rewh}.
Due to the symmetry of the Kronecker coefficients we are assuming
$\nu_2 \le \mu_2$.
We have to consider
three cases: (a) $\la_2 -\la_3 \le \nu_2 -2\la_3$,
(b) $\nu_2 -2\la_3 < \la_2 -\la_3 \le \mu_2 -2\la_3$ and
(c) $\mu_2 - 2\la_3 < \la_2 -\la_3 $.
In the first two cases Remmel-Whitehead formula yields the same
formula for ${\sf k}(\wt\la,\wt\mu,\wt\nu)$.
So, we have only two cases to consider:
(1) $\la_2 +\la_3 \le \mu_2$ and (2) $\mu_2 < \la_2 +\la_3$.
In the first case Theorem~\ref{teor:rewh} yields
\[
{\sf k}(\wt\la,\wt\mu,\wt\nu) = (y^\prime-x^\prime)(y^\prime \ge x^\prime)
\]
where
$x^\prime = \max\left( 0, \left\lceil
\tfrac{\nu_2 - 2\la_3 + \la_2 -\la_3 + \mu_2 -2\la_3 -(m-4\la_3)}{2}\right\rceil
\right)$ and
$y^\prime = \left\lceil \tfrac{\nu_2 - 2\la_3 + \la_2 - \la_3
- (\mu_2 - 2\la_3) +1}{2}\right\rceil $.
It is straightforward to check that $x^\prime =x$ and $y^\prime = y$,
so the first claim follows.

The second case is similar.
\end{proof}

\end{document}